%% file: hdivisor.tex
\documentclass[12pt]{amsart}
\allowdisplaybreaks[1]
\usepackage{enumerate}
\usepackage{graphicx}

\usepackage{amsmath}
\usepackage{amssymb}
\usepackage{amsfonts}
\usepackage{verbatim}
\usepackage[dvipsnames]{xcolor}
\usepackage{hyperref}
\allowdisplaybreaks
\makeindex

 \newtheorem{theorem}{Theorem}[section]
 \newtheorem{corollary}[theorem]{Corollary}
 \newtheorem{lemma}[theorem]{Lemma}

\newtheorem{observation}[theorem]{Observation}
\theoremstyle{definition}

\theoremstyle{remark}

\newtheorem{fact*}{Fact}

\DeclareMathOperator{\Aut}{Aut}
\DeclareMathOperator{\divisor}{div}

\newcommand{\tr}{\operatorname{tr}}
\newcommand{\rk}{\operatorname{rk}}

\newcommand{\til}{\raise.17ex\hbox{$\scriptstyle\mathtt{\sim}$}}

\newcommand\beq{\begin{equation}}

\newcommand\eeq{\end{equation}}

\newcommand{\bbm}{\left[ \begin{smallmatrix}}
\newcommand{\ebm}{\end{smallmatrix} \right]}
\newcommand{\bpm}{\left( \begin{smallmatrix}}
\newcommand{\epm}{\end{smallmatrix} \right)}
\numberwithin{equation}{section}

\newlength{\Mheight}
\newlength{\cwidth}

\newcommand{\dfn}[1]{{\bf #1}\index{#1}}

\title[Free principal divisors]{Free noncommutative principal divisors and commutativity of the tracial fundamental group}
\author[J. E. Pascoe]{
J. E. Pascoe
}
\address{Department of Mathematics\\
1400 Stadium Rd\\
  University of Florida\\
 Gainesville, FL 32611}
\email[J. E. Pascoe]{pascoej@ufl.edu}

\date{\today}


\begin{document}

\begin{abstract}
We define the principal divisor of a free noncommuatative function. We use these divisors to compare the determinantal singularity sets of free noncommutative functions.
We show that the divisor of a noncommutative rational function is the difference of two polynomial divisors.

We formulate a nontrivial theory of cohomology, fundamental groups and covering spaces for tracial free functions.
We show that the natural fundamental group arising from analytic continuation for tracial free functions is a direct sum of copies of $\mathbb{Q}$.
Our results contrast the classical case, where the analogous groups may not be abelian, and the free case, where free universal monodromy implies 
such notions would be trivial.
\end{abstract}

\maketitle

\tableofcontents

\input{divisor}
\input{module}

\input{full}

\input{questions}

\bibliography{references}
\bibliographystyle{plain}

\end{document}

%% file: divisor.tex
\section{Introduction}
	Our first goal is to understand equality  the determinants of free noncommutative function and of their zero sets.
	Consider the two noncommutative polynomials:
		$$1+xy, 1+yx.$$
	Evidently, $\tr[(1+XY)^n] = \tr[(1+YX)^n]$ for all $n$ when $X$ and $Y$ are like-size matrices, and therefore $\det 1+XY = \det 1+YX.$
	(This is the Weinstein-Aronszajn identity, which is also known as the Sylvester determinant theorem.)
	Second, consider the two expressions:
		$$e^{x+y}, e^{x}e^{y}.$$
	Evidently, $\det e^{X+Y} = \det e^{X}e^{Y}$ when $X$ and $Y$ are like-size matrices, as generally $\det e^T=e^{\tr T},$ 
	and $$\det e^{X+Y} = e^{\tr X+Y} = e^{\tr X + \tr Y} = e^{\tr X}e^{\tr Y}= \det e^{X}e^{Y}.$$
	If one can write
		$e^Xe^Y = e^Z$ for some $Z$ 
	this implies $$Z=X+Y + \textrm{ commutators}$$
	which is the qualitative aspect of the Baker-Campbell-Hausdorff formula.

	We characterize the equality of determinantal zero sets in terms of \emph{principal divisors} corresponding to a
	noncommutative function. Divisors are concretely given by the logarithmic derivative $\det f$ which is itself
	naturally interpreted as a $g$-tuple of noncommutative functions which we will denote by $\divisor f$.

	We also analyze ``homotopy" and cohomological theory for such functions and their domains, motivated by the basic question: which functions are divisors?
	We define the tracial covering space to be the set of paths in a domain originating at some fixed base point which can be distinguished via analytic continuation,
	and show that the corresponding group generated by loops, the \emph{tracial fundamental group} is abelian.
	Prior developments in sheaf theory were made for free noncommutative functions in Klep, Volcic, Vinnikov \cite{KVVgerms}, in the study of noncommutative symmetric functions
	by Agler and Young \cite{aysym}, the noncommutative manifold theory Agler, McCarthy and Young \cite{AMYMan}, implicit function theorems of Agler and McCarthy \cite{AMImp},
	and free universal monodromy \cite{pascoeplush}.
	
	The search for topological and geometric invariants for free domains is independently motivated by the desire to understand their change of variables theory, which is well developed for 
	maps between convex sets \cite{HKMBook, hkms09, helkm11, augatPLMS, ahkmMann, ahkmjfa, pascoeMathZ}, which itself was motivated by simplifying dimension-free systems of matrix inequalities arising in
	systems engineering \cite{CHSY,DeOliveira2009, SIG}.

\section{Noncommuative functions and their divisors}
Define the \dfn{$d$-dimensional matrix universe} to be
	$$\mathcal{M}^d = \bigcup M_n(\mathbb{C})^d.$$
Say $D\subseteq \mathcal{M}^d$ is a \dfn{free set} if
\begin{enumerate}
	\item $X, Y \in D \Rightarrow X \oplus Y \in D,$
	\item $X \in D,$ $U$ unitary $\Rightarrow U^*XU \in D.$  (Here, $U^*(X_1, \ldots, X_d)U = (U^*X_1U,\ldots, U^*X_dU)$.)
\end{enumerate}
Let $D_n = D \cap M_n(\mathbb{C})^d.$
We say $D$ is \dfn{differentiable} if for each $n,$ $D_n$ is an open $C^1$ manifold such
that the complex span of the tangent space at any point $X\in D_n$ is equal to all of $M_n(\mathbb{C})^d$.
For example, the set of all $d$-tuples of self-adjoints is differentiable. 
We say $D$ is \dfn{open} if each $D_n$ is open. We say $D$ is \dfn{connected} if each $D_n$ is connected. We say $D$ is \dfn{simply  connected} if each $D_n$ is simply connected.
Let $D\subseteq \mathcal{M}^d$ be a free set. We say $f: D\rightarrow \mathcal{M}^{\hat{d} \times \hat{e}}$ is a \dfn{free function}
if
\begin{enumerate}
	\item $f(X \oplus Y) = f(X)\oplus f(Y)$
	\item $f(U^*XU) =U^*f(X)U$ for $U$ unitary.
\end{enumerate}
We view the outputs of a free function mapping into $\mathcal{M}^{\hat{d} \times \hat{e}}$ as block $\hat{d}$ by $\hat{e}$ matrices.
We say $f: D\rightarrow \mathbb{C}$ is a \dfn{determinantal free function}
if
\begin{enumerate}
	\item $f(X\oplus Y) = f(X) f(Y),$
	\item $f(U^*XU)=f(X)$ for $U$ unitary.
\end{enumerate}
We say a free function or determinantal free function is $C^1$ if it is defined on a differentiable free set $D$ and the function is $C^1$ on each $D\cap M_n(\mathbb{C}^d).$

We say $f: D\rightarrow \mathbb{C}$ is a \dfn{tracial free function}
if
\begin{enumerate}
	\item $f(X\oplus Y) = f(X)+ f(Y),$
	\item $f(U^*XU)=f(X)$ for $U$ unitary.
\end{enumerate}
Given a differentiable tracial free function $f$, we define the \dfn{free gradient} to be the unique free function $\nabla f$ such that
$$\tr H \cdot \nabla f(X) = Df(X)[H]$$
where $n$ is the size $X.$ Here, $(A_1, \ldots, A_d) \cdot (B_1, \dots, B_d) = \sum A_iB_i.$ Given a determinantal free function $f$, we may induce a tracial free function $\log {f}$ on its nonsingular set.
For any square free function $f,$ we may induce a determinantal free function $\det f.$   
Let $f$ be a differentiable determinantal free function such that $f\not\equiv 0.$ We define the \dfn{principal divisor} corresponding to $f,$ denoted $\divisor f $
by $$\divisor f = \nabla \log f.$$
That is, $\divisor f$ is
the unique free function $g:D\setminus \{X| f(X)=0\}\rightarrow \mathcal{M}^d$ such that
	$$D\log f (X)[H]=\tr \sum^d_{i=1} H_ig_i(X).$$
We overload the divisor notation for free functions: for a free function $f,$ $\divisor f$ is defined to be equal to $\divisor \det f.$
Note that
	$$\divisor fg = \divisor f + \divisor g.$$

The following essentially trivial set of observations is the reason we want to consider divisors.
\begin{observation} \label{critical}
	Let $f, g$ be a $C^1$ nonzero determinantal free functions.
	\begin{enumerate}
		\item There exists an invertible locally constant determinantal free function $c$ such that $f=cg$ if and only if $\divisor f = \divisor g,$
		\item $f/g$ has a $C^1$ extension to the whole domain if and only if there is a $C^1$ determinantal function $h$ on the domain such that  $\divisor f - \divisor g= \divisor h$,
		\item $f/g$ and $g/f$ have a $C^1$ extension to the whole domain if and only if $\divisor f - \divisor g$ has a continuous extension to the whole domain.
	\end{enumerate}
\end{observation}
We note that the technique of viewing matrix identities, in our case coming from equality of principal divisors, as derivatives of determinantal identities was previously use by Terrence Tao in \cite{tao_2013}.

For example, $\divisor e^Xe^Y = \divisor e^X + \divisor e^Y = (1,1).$
Also note, $\divisor e^{X+Y} = (1,1).$ Therefore, $\divisor e^Xe^Y = \divisor e^{X+Y}.$
Notably, this gives a proof of the qualitative aspect of the Baker-Campbell-Hausdorff formula used in Lie theory, the difference between $\log e^Xe^Y$ and $\log e^{X+Y}$ has trace zero and is thus a sum of commutators.

Computing the divisor of a general free function is made possible via the following computation.
\begin{theorem}
	Let $f: D\rightarrow \mathcal{M}^{\hat{d} \times \hat{d}}$ be a $C^1$ free function such that $\det f \not\equiv 0$.
	$\tr H\cdot \divisor f = \tr Df(X)[H]f(X)^{-1}.$
\end{theorem}
\begin{proof}
Jacobi's formula states that the derivative of the determinant  of $\det A$ is given by:
	$D \det A [K] = \tr K\textrm{adj } A.$
Now we compute $D \det f(X) [H]$ using the chain rule
	$$D \log\det f(X)[H] = \frac{1}{\det f(X)}\tr Df(X)[H]\textrm{adj } f(X) = \tr Df(X)[H]f(X)^{-1}.$$
\end{proof}
Consider $\divisor 1+XY.$ Note
	\begin{align*} 
	\tr H \cdot \divisor 1+XY & = \tr (H_1Y+XH_2)(1+XY)^{-1} \\
	& = \tr H_1Y(1+XY)^{-1}+H_2(1+XY)^{-1}X
	\end{align*}
Therefore, $$\divisor 1+XY = (Y(1+XY)^{-1},(1+XY)^{-1}X).$$
A similar computation gives that,
	$$\divisor 1+YX = ((1+YX)^{-1}Y,X(1+YX)^{-1}).$$
We now show $\divisor 1+XY =  \divisor 1+YX.$ 
Note $Y(1+XY) = (1+YX)Y,$ so $Y(1+XY)^{-1} = (1+YX)^{-1}Y.$ The argument for the second coordinate is similar.

\section{Divisors of polyonmials and rational functions}
A \dfn{free polynomial} is a syntactically valid expression in terms of some noncommuting indeterminants, $+,$ multiplication and scalar coefficients.
For example,
	$$1-xy, xy-yx, 37zxy+8xyxz, x+y$$
are all free polynomials. 
A \dfn{square matricial free polynomial} is a square matrix with free polynomial entries.
For example,
	$$\bpm 1-xy & xy-yx \\ 37zxy+8xyxz & x+y\epm,$$
is a square matricial free polynomial.
A \dfn{free rational expression} is a syntactically valid expression in terms of some noncommuting indeterminants, $+,$ multiplication, parentheses, inverses and scalar coefficients.
A rational expression $r$ is \dfn{nondegenerate} if there exists at least one $X \in \mathcal{M}^d$ such that $r(X)$ is well-defined.
For example,
	$$(xy-xy)^{-1}, (x^2y+78(xy-yx)^{-1}), (y(1-xy)^{-1}-(1-yx)^{-1}y)^{-1}$$
are all free rational expressions, but $(y(1-xy)^{-1}-(1-yx)^{-1}y)^{-1}$ is not nondegenerate.
A \dfn{square matricial rational expression} is a square matrix with free rational expression entries.
For any nondegenerate rational expression there exists an affine linear square matricial polynomial $L$ and constant rectangular $b, c$
such that $r=b^*L^{-1}c$ and $L^{-1}$ is well-defined whereever the expression $r$ was defined \cite{PMCohn}. (Note that this may not be the maximal domain the function continues to \cite{volcic}.)
\begin{theorem}
	Let $f, g$  square matricial free polynomials such that $\det f,\det g \not\equiv 0$.
	Suppose $\det f / \det g$ and $\det g/\det f$ are entire.
	Then,
		$\divisor f = \divisor g.$
\end{theorem}
\begin{proof}
	Note that, as functions on $M_n(\mathbb{C})$ $\det f / \det g$ and $\det g/\det f$ are polynomials and reciprocals, and hence constant. Therefore, $\divisor f = \divisor g$
	by Observation \ref{critical}.
\end{proof}
\begin{lemma}
	Let $r$ be a nondegenerate square matricial free rational expression.
	There exist square matricial free polynomials $p, q$ such that
		$$\det r= \det p / \det q = \det pq^{-1}.$$
\end{lemma}
\begin{proof}
	Let $r = b^*L^{-1}c,$ where $L$ is affine linear.
	Now consider
		\begin{align*}\det \bpm L & c \\ -b^* & 0 \epm /\det L &= \det \bpm L & c \\ -b^* & 0 \epm \det \bpm L^{-1} & 0 \\ 0 & 1 \epm \\
		&= \det \bpm 1 & c \\ -b^*L^{-1} & 0 \epm\\
		&= \det b^*L^{-1}c
		\end{align*}
	where the last line uses the fact that $\det \bpm A & B \\ C & D \epm = \det A \det D-CA^{-1}B.$
	Therefore, $$\det r= \det \bpm L & c \\ -b^* & 0 \epm /\det L.$$
		
\end{proof}

As a corollary of the preceeding lemma, we see the following theorem as an immediate corollary. Determinantal singularity and polar sets of free rational functions are no more complicated 
than those of free polynomials.
\begin{theorem} \label{ratl}
	Let $r$ be a nondegenerate square matricial free rational expression such that $\det r \not\equiv 0$.
	There exist square matricial free polynomials $p, q$ such that 
		$$\divisor r = \divisor p - \divisor q.$$
\end{theorem}
For example, consider Schur complement $D-BA^{-1}C$ as a rational function.
Note $$\det \bpm A & B \\ C & D \epm = \det A \det D-CA^{-1}B.$$
That is, $$\divisor D-CA^{-1}B=\divisor \bpm A & B \\ C & D \epm - \divisor A.$$
It is somewhat interesting to explicitly compute the divisors here and equate like terms, so we shall do so.
First,
	$$\divisor D-CA^{-1}B = \bpm
		A^{-1}B(D-CA^{-1}B)^{-1}CA^{-1} &
		-A^{-1}C(D-CA^{-1}B)^{-1}\\
		-(D-CA^{-1}B)^{-1}A^{-1}B &		
		(D-CA^{-1}B)^{-1}	
	\epm$$
Second, note that
	$$
		\tr D\bpm A & B \\ C & D \epm\bbm H_{11} & H_{12} \\ H_{21} & H_{22} \ebm \bpm A & B \\ C & D \epm^{-1} =
		\tr \bpm H_{11} & H_{12} \\ H_{21} & H_{22} \epm\bpm A & B \\ C & D \epm^{-1}.
	$$
Therefore,
	$$ \divisor \bpm A & B \\ C & D \epm =
		\bpm A & B \\ C & D \epm^{-1}$$
Finally, $$\divisor A = \bpm A^{-1} & 0 \\0 &0\epm.$$
Equating like terms, we see that 
	$$\bpm A & B \\ C & D \epm^{-1} = \bpm
		A^{-1} + A^{-1}B(D-CA^{-1}B)^{-1}CA^{-1}&
		-A^{-1}C(D-CA^{-1}B)^{-1}\\
		-(D-CA^{-1}B)^{-1}A^{-1}B &		
		(D-CA^{-1}B)^{-1}	
	\epm$$
which is the classical block two by two matrix inversion formula.

Let $D$ be a free set.
Let $f$ be a $C^1$ determinantal free function defined on a dense subset of $D$.
We say $\divisor f$ is \dfn{effective} if $f$ extends to a $C^1$ function on all of $D.$
Say $\divisor f \preceq \divisor g$ if $\divisor f - \divisor g$ is effective.
Define the zero set of a free function $f$ to be $\mathcal{Z}(f)=\{X\in D| \det f=0\}.$

Say a square matricial free polynomial $p$ is \dfn{atomic} if $\det p \not\equiv 0$ and if $p_1p_2 = p,$ then either $\det p_1$ or $\det p_2$ is locally constant. 
Helton, Klep and Vol\v{c}i\v{c} show that if $q = p_1\ldots p_n$ where each $p_i$ is an atomic square matricial free polynomial, and $p$ is another atomic square
matricial free polynomial such that $\mathcal{Z}(p)\subseteq\mathcal{Z}(q),$ then $\det p = \det cp_i$ for some $i$ and some nonzero constant $c$ \cite[Theorem 2.12]{hkvfactor}.
That is, up to equivalence of principal divisors, the factorization of $q$ is unique.
Combining this with Theorem \ref{ratl}, we see that the additive structure of the rational divisors is free abelian, generated by the atomic matricial square free polynomials.  
\begin{corollary}
	The set of principal divisors of nondegenerate square matricial free rational expressions such that $\det r \not\equiv 0$ with addition is a free abelian group.
\end{corollary}

Moreover, we conclude the following theorem about inclusion of zero sets.
\begin{theorem}
	Let $p, q$ be square matricial free polynomials, viewed as entire functions, such that $\det p,\det q \not\equiv 0$.
	There exists an $n\in \mathbb{N}$ such that $\divisor p \preceq n \divisor q$
	if and only if
	$$\mathcal{Z}(p)\subseteq\mathcal{Z}(q).$$
\end{theorem}
\begin{proof}
	Without loss of generality, we will assume $p$ is atomic.
	Write $\divisor q = \sum n_i\divisor p_i$  where $p_i$ are atomic, $n_i \in \mathbb{N}.$
	Note that $$\mathcal{Z}(q)=\mathcal{Z}(p_1p_2\ldots p_n).$$
	The claim is that $p \preceq q$ is equivalent to saying that one of the $\divisor p_i$ is equal to $\divisor p,$ which is essentially the 
	Helton-Klep-Vol\v{c}i\v{c} theorem stated above from  \cite[Theorem 2.12]{hkvfactor}.
\end{proof}

\section{The tracial fundamental group}\label{tracialfund}

Say an analytic free map $g: D \rightarrow \mathcal{M}^d$ is \dfn{closed} if $ \tr K\cdot Dg(X)[H] = \tr H\cdot Dg(X)[K].$
Say a free map $g: D \rightarrow \mathcal{M}^d$ is \dfn{exact} if  there exists a function $f$ such that $ \nabla f = g.$
We define the \dfn{first tracial free cohomology group} $H^1_{\tr}(D)$ to be the vector space of closed free functions modulo the exact free functions.


Let $D$ be a connected open free set.
Say $D$ is \dfn{anchored} if it contains a nonempty open simply connected free subset $B.$ (We call $B$ an \dfn{anchor} of $D.$)
We say tracial free function $f: B \rightarrow \mathbb{C}$ is a \dfn{global germ} if it analytically continues along every path in $D.$
Let $X, Y \in D.$
We say a path $\gamma$ \dfn{essentially takes} $X$ to $Y$ whenever 
\begin{enumerate}
	\item $\gamma(0) = X^{\oplus k},$ for some $k\in \mathbb{N},$
	\item $\gamma(1) = Y^{\oplus l},$ for some $l\in \mathbb{N}.$
\end{enumerate}
For set theoretic reasons, we view a path coupled with its terminal endpoint $Y,$ as there might be multiple ways to break down $Y^{\oplus l}$
as a direct sum.
Given a global germ $f$ and $X\in B$ an anchor, $\gamma$ a path essentially taking $X$ to $Y$, we can define $f(\gamma)$ to be the value
of $f$ analytically continued along $\gamma$ multiplied by the normalizing factor $\frac{1}{l}.$
Given two paths $\gamma_1$ essentially taking $X$ to $Y$, $\gamma_2$ essentially taking $Y$ to $Z,$ we define their \dfn{concatenation product} $\gamma_2\gamma_1$ to be 
	$$\gamma_2\gamma_1 (t) = 
	\begin{cases}
                                   \gamma_1^{\oplus m_2}(2t) & \text{if $0\leq t<1/2$} \\
                                   \gamma_2^{\oplus m_1}(2t-1) & \text{if $1/2 \leq t \leq 1$}
  	\end{cases}$$
where $m_1$ and $m_2$ are chose to make the sizes compatiable.

Say two paths $\gamma_1, \gamma_2$ essentially taking $X\in B$ an anchor to $Y$ in $D$ are \dfn{trace equivalent} if
	\begin{enumerate}
		\item $\gamma_1$ and $\gamma_2$ both essentially take $X$ to $Y$
		\item for
		every global germ $f$ and path $\gamma_3$ essentially taking $Y$ to $Z$, $f(\gamma_3\gamma_1) = f(\gamma_3\gamma_2).$
	\end{enumerate}
We let $\gamma^X$ denote the trivial path from $X$ to $X.$
Note that as a path essentially taking $X$ to $Y,$  the path $\gamma^{\oplus m}$ is equivalent to $\gamma.$ (Any two endpoint homotopic paths will also be trace equivalent.)

Given an anchored free set $D,$ fix an anchor $B$ and an 
\dfn{anchor point} $X \in B.$
We define the \dfn{tracial fundamental group} $\pi_1^{\tr}(D)$ to be the set of trace equivalence classes of paths essentially taking $X$ to $X.$
(As a group, the definition of the fundamental group is independent of our choice of $X$ because $D$ is connected.)

We define the \dfn{tracial covering space} $C^{\tr}(D)$ to be the set paths up to trace equivalence essentially taking $X$ to a point $Y \in D.$

There is a natural covering map $\rho:C^{\tr}(D) \rightarrow D$ taking a path to its terminal endpoint $Y.$
There is a natural inclusion map $\iota: B \rightarrow C^{\tr}(D)$ taking $Y \in B$ to a path in $B$ essentially taking $X$ to $Y.$
That is, we view an anchor $B \subseteq D$ as a subset of $C^{\tr}(D)$ by identifying it with the set of points over $B$ in the component of the trivial path.
The tracial fundamental group acts naturally on the tracial covering space.
Any global germ induces an analytic function on the tracial covering space.

We define a distance function $d:C^{\tr}(D)\times C^{\tr}(D) \rightarrow \mathbb{R}\cup\{+\infty\}.$
Let $\gamma_1$ be path essentially taking $X$ to $Y$ and $\gamma_1$ be path essentially taking $X$ to $Z.$
We set $d(\gamma_1,\gamma_2)=\infty$ if $Y$ and $Z$ are different sizes.
Let $\Gamma_{Y,Z}$ be the set of paths taking $Y$ to $Z.$ (Specifically, this is a path in $D_n$ where $n$ is the size of $Y$ and $Z.$) Let $\|\gamma\|$ be the length of $\gamma.$
We set $d(\gamma_1,\gamma_2)=\inf_{\gamma\in \Gamma_{Y,Z}, \gamma\gamma_1 = \gamma_2} \|\gamma\|.$

Given $g: D \rightarrow \mathcal{M}^d$ which is closed, we can induce a function $f:C^{\tr}(D)\rightarrow \mathbb{C}$ by first defining an $f$ on $B$
and then defining $f$ globally by analytically continuing along every path.

Let $B \subset D$ be open free sets, $D$ connected.
The \dfn{free universal monodromy theorem} says that if a free function on $B \subseteq D$ analytically continues along any path in $D$
then it has a unique global analytic continuation \cite{pascoeplush}.
The free gradient of a global germ analytically continues along any path, therefore given a global germ $f: B \rightarrow \mathbb{C},$
the function $\nabla f$ analytically continues uniquely to all of $D.$
Therefore, we see the following key observation.
\begin{observation}
	Let $D$ be an open connected free set with anchor $B$.
	Let $f: B \rightarrow \mathbb{C}$ be a global germ on $D.$
	Let $\gamma \in \pi_1^{\tr}(D).$
	Then, the function
		$f(\gamma'\gamma) - f(\gamma')$
	on $C^{\tr}(D)$ is locally constant.
	If the anchor point $X \in D_n,$ define $c^f_\gamma = \frac{f(\gamma) - f(\gamma_X)}{n}.$
	Note that $c^f_{\gamma}$ depends only on the cohomology class of $\nabla f$ in $H^1_{\tr}(D).$
\end{observation}
We now note that $\phi_g: \gamma \mapsto c^f_\gamma$ is a homomorphism into $\mathbb{C},$ where $g$ is the cohomology class of $\nabla f,$ because
\begin{align*}
	c^f_{\gamma_1\gamma_2}&=f(\gamma_1\gamma_2) - f(\gamma_X) \\
	&=f(\gamma_1\gamma_2) -f(\gamma_1) + f(\gamma_1) - f(\gamma_X)\\
	&=c^f_{\gamma_2} + c^f_{\gamma_1}.
\end{align*}
Moreover, for every nontrivial $\gamma \in \pi_1^{\tr}(D)$ there exists a tracial cohomology class $g\in H^1_{\tr}(D)$ whose analytic continuation along $\gamma$
is nontrivial, and therefore $\pi_1^{\tr}(D)$ is a torsion free abelian group, as the map
$$\prod_{g\in H^1_{\tr}(D)} \phi_g: \pi_1^{\tr}(D) \rightarrow \prod_{g\in H^1_{\tr}(D)}\mathbb{C}$$
is an injective homomorphism.
Note that, as $f(\gamma_X^{\oplus k}\oplus\gamma) = \frac{1}{k+1}f(\gamma)+\frac{k}{k+1}f(\gamma_X)$
for a global germ, we see that $(k+1)\gamma_X^{\oplus k}\oplus\gamma \equiv \gamma$. Therefore, the tracial fundamental group must be divisible.
\begin{theorem} \label{abelian}
	Let $D$ be an anchored free set.
	The tracial fundamental group $\pi_1^{\tr}(D)$ is a countable torsion free divisible abelian group, that is it is isomorphic to a direct sum of copies of $\mathbb{Q}$.
\end{theorem}
The situation of tracial functions on free sets contrasts significantly with that for classical domains.
For example, for connected domains $D \subseteq \mathbb{C}$ any two nonhomotopic paths can be distinguished by analytic continuation. (The uniformization theorem says
that the covering space is conformally equivalent to the disk or the whole complex plane.) Namely, for the two punctured plane, the fundamental group is nonabelian.
The tracial free situation behaves more like the situation for Lie groups, which have abelian fundamental group due to their rigid analytic structure.

Consider the domain
	$$GL=\{X| \det X \neq 0\}.$$
The free gradient of a global germ on $GL$ is a free function on $GL$ and therefore has a Laurent series.
Therefore, the failure of monodromy is witnessed by the function $\log z$ alone.
	$$\pi^{\tr}_1(GL) \cong \mathbb{Q}.$$
One might suspect that the tracial covering space $C^{\tr}(D)$ is $\mathcal M^1$ with the covering map given by $e^Z.$ It is not, as such functions would not have the required
periodicity modulo constants.
In fact, the path
	$$\gamma(t)= \bpm e^{2\pi it} & 1\\
	0	& e^{-2\pi it}
	\epm$$
has no preimage under the exponential map. 
However, the $\tr \log \gamma(t)$ is constant along the path, so the path is equivalent to the trivial path. (In fact, it is not classically homotopy equivalent to the trivial path
and, therefore, the covering space $C^{\tr}(D)$ is not simply connected.)
Heuristically, the tracial covering space adds some points at infinity to $\mathcal M^1_2.$
Generally, the tracial covering space has to be small enough and complicated enough to induce Liouville theorem type behavior inducing the periodicity modulo constants coming from free universal monodromy of the free gradient, that is, to allow the tracial fundamental group to be abelian.

As a group $GL_m(\mathbb{C})$ classically has fundamental group $\mathbb{Z}.$ We leave an exercise to the reader that the tracial fundamental group of invertible block $k$ by $k$ matrices is isomorphic to $\mathbb{Q}.$

Similarly, given $\Lambda \subseteq \mathbb{C}$ a finite set, define the domain
	$$G_\Lambda = \{X| \det X - \lambda \neq 0 \textrm{ for all } \lambda \in \Lambda\}.$$
An exercise shows that 
	$$\pi^{\tr}_1(G_{\Lambda}) \cong \mathbb{Q}^{|\Lambda|}.$$



Recall the notion of \dfn{rank of an abelian group}, the maximal size of a set $z_1, \ldots, z_n$ such that there are no nontrivial relations of the form
$\sum n_iz_i =0,$ where the $n_i$ are integers. (That is, the maximal size of a \dfn{linearly independent} subset.)
\begin{theorem}
	Let $D$ be an anchored free set.
	The following are true:
	\begin{enumerate}
		\item $\dim H^1_{\tr}(D) \neq 0$ if and only if  $\rk \pi_1^{\tr}(D)\neq 0,$
		\item $\dim H^1_{\tr}(D) \leq \rk \pi_1^{\tr}(D)$ whenever both quantities are at most countably infinite. 
	\end{enumerate}
\end{theorem}
\begin{proof}
	(1) follows from the fact that if $\dim H^1_{\tr}(D) =0$ then monodromy holds.

	(2) Without loss of generality, suppose $\rk \pi_1^{\tr}(D)$ is finite.
	Choose a maximal linearly independent set in $\pi_1^{\tr}(D)$ denoted $\gamma_1, \gamma_2, \ldots, \gamma_N.$
	Just suppose there were a linearly independent set
		$g_1, \ldots, g_{N+1}$
	in $H^1_{\tr}(D).$
	The matrix 
		$[\phi_{g_{i}}(\gamma_j)]_{i,j}$
	is singular. Let $(\alpha_1,\ldots, \alpha_{N+1})$ be the kernel vector.
	Write $g= \sum \alpha_i g_i.$
	Note,
		$g(\gamma_j) = 0$ for all $j$
	and therefore is equal to $0,$ which is a contradiction.
\end{proof}
The Hurewicz theorem says that the classical first homology group is equal to the abelianization of the fundamental group.
That is, for a path-connected manifold $F,$ $\pi_1(F)/[\pi_1(F),\pi_1(F)] \cong H_1(F).$

Let $D$ be an anchored free set such that each $D_n$ is nonempty.
Let $B$ be an anchor  such that each $B_n$ is nonempty.
Let $X\in D_1$ be our base point.
There is a \dfn{natural gradation} of $\pi^{\tr}_1(D)$, define $\pi^{\tr}_1(D)_n$ to be those equivalence classes of paths inside $D_n.$
Note that $\pi^{\tr}_1(D)_n$ is isomorphic to a quotient of a subgroup of $\pi_1(D_n)$, and because it is abelian by Theorem \ref{abelian}, a quotient of a subgroup of $H_1(D_n).$
Moreover, $\pi^{\tr}_1(D)_n$ is a subgroup of $\pi^{\tr}_1(D)_{nm}.$
Therefore, $\pi^{\tr}_1(D)$ is the direct limit of  $\pi^{\tr}_1(D)_{n!}$ and under isomorphism, a series of quotients of $H_1(D_{n!}).$

The natural gradation gives a ``coarse approximation" $\pi^{\tr}_1(D)$ as a limit of quotients of $H_1(D_{n!})$ which can be helpful in computing qualitative properties 
of $\pi^{\tr}_1(D).$
For example, if $\rk H_1(D_{n!})$ is uniformly bounded by $k,$ then $\rk \pi^{\tr}_1(D) \leq k.$

We define the \dfn{integral tracial fundamental group} $\pi_1^{\tr}(D)_\mathbb{Z}$ to be $\sum n\pi^{\tr}_1(D)_{n}.$
We define the \dfn{integral tracial cohomology group} $H^1_{\tr}(D)_\mathbb{Z}$ to be the set $g\in H^1_{\tr}(D)$ such that the range of $\phi_g(\gamma)$ is contained in $\mathbb{Z}$
for $\gamma \in \pi_1^{\tr}(D)_\mathbb{Z}.$
We see the following characterization of divisors.

\begin{theorem}
	Let $D$ be an anchored free set such that each $D_n$ is nonempty.
	Let $B$ be an anchor  such that each $B_n$ is nonempty.
	Let $g: D \rightarrow \mathcal{M}^d$ be an analytic free function.
	There exists a nonsingular determinantal function $f$ such that $\divisor f = g$
	if and only if the tracial cohomology class of $\frac{1}{2\pi i}g$ is contained in $H^1_{\tr}(D)_\mathbb{Z}.$
\end{theorem}
\begin{proof}
	First suppose the tracial cohomology class of $\frac{1}{2\pi i}g$ is contained in $H^1_{\tr}(D)_\mathbb{Z}.$
	Without loss of generality, $0 \in B.$
	We may define a tracial free function $f: B \rightarrow \mathbb{C}$ such that on $B,$ $\divisor e^f = g.$
	Let $\gamma_1, \gamma_2$ be two different paths taking $0^{\oplus n}\in B_n$ to $Y \in B_n$
	So $\gamma= \gamma_1^{-1}\gamma_2$ is a path essentially taking $0$ to $0.$ Note $\gamma \in \pi^{\tr}_1(D)_n.$
	So $f(\gamma)- f(\gamma_0) = \phi_g(\gamma).$
	By assumption $n\phi_g(\gamma) \in 2\pi i\mathbb{Z}.$
	Therefore, now viewing $\gamma$ as a path essentially taking $0$ to $0^{\oplus n}$ we see $f(\gamma)-nf(0) \in 2\pi i\mathbb{Z}.$
	Thus, $e^{f(\gamma)} = e^{nf(0)},$ so we are done.
	
	To see the converse, suppose not. Then there exists a path $\gamma$ in $D_n$ essentially taking $0$ to $0$
	such that $n(f(\gamma)- f(\gamma_0)) = n\phi_g(\gamma) \notin 2\pi i\mathbb{Z}.$
	Therefore, now viewing $\gamma$ a path essentially taking $0$ to $0^{\oplus n}$ we see that that
	$e^{f(\gamma)} \neq e^{nf(0)}.$
\end{proof}
We could take the approach of defining close and exact forms with respect to divisors of  determinantal free functions. Similarly, we could define a fundamental group via paths which are distinguished
by analytic continuation. One obtains $\pi^{\det}_1(D) = \pi_1^{\tr}(D)$ and $H^1_{\det}(D) = H^1_{\tr}(D) / (2\pi i)H^1_{\tr}(D)_\mathbb{Z}.$
The above theorem says that a $g: D \rightarrow \mathcal{M}^d$ is a principal divisor if and only if its cohomology class in $H^1_{\det}(D)$ is zero. On face, the structure of the first determinantal cohomology group is
significantly more complicated than the first tracial cohomology group.




%% file: module.tex
\section{The automorphisms of a module repecting a homomorphism}	
	We can clear away the superfluous data from our proof of the commutativity of the tracial fundamental group to obtain the following remarks 
	on the level of modules. These remarks provide a roadmap for potential generalizations of the theory laid out above.

	Let $F$ be an $R$-module. Let $E$ be a submodule of $F.$
	We define the automorphisms of $F$ over $E,$
	denoted $\Aut(F, E)$ to be the group of automorphisms
	of $F$ fixing $E.$ 
	Let $H$ be an $R$-module and
	$\varphi: F\rightarrow H$ be a module homomorphism.
	We define the automorphisms of $F$ over $E$ respecting
	$\varphi,$ denoted $\Aut(F, E, \varphi)$
	to be the subgroup of elements $\gamma\in\Aut(F, E)$
	such that $\ker \varphi$ is fixed and
	$\varphi(\gamma \cdot f) = \varphi(f).$

	For example, let $F$ be the space of multivalued-analytic functions on $\mathbb{C}\setminus\{0\}$ of the form
		$$c_{\log{}} \log{z} + \sum^{\infty}_{n=-\infty} c_n z^n.$$
	Let $E$ be the subspace such that $c_{\log{}} = 0.$ (Therefore, all the functions in $E$ are single-valued.)
	Let $H=E.$ The map $\varphi(f) = f'$ takes $F$ to $H.$
	The action of an automorphism $\gamma$ in $\Aut(F, E)$ is determined by its action on $\log z.$
	Namely, $\gamma \log z = \log z + e$ for some function $e \in E.$
	The action of an automorphism $\gamma$ in $\Aut(F, E, \varphi)$ needs to satisfy the extra property that 
	$\varphi(\gamma \cdot \log z) = \varphi(\log z).$ So, $1/z + e' = 1/z$ and consequently, $e'=0$, so $e$ is constant.
	By our inspection,  $\Aut(F, E, \varphi) \cong \mathbb{C}.$

	\emph{In our applications, $F$ was the space of germs based at some point $X$ which essentially analytically continue along every path, $E$
	was the subspace of those corresponding to globally-defined analytic functions, $H$ was a space of analytic functions on a domain satisfying monodromy
	and $\varphi$ was a derivative-like map, either the free gradient or the divisor map. Specifically, $H$ was the space of analytic free functions on some free set, which are always
	monodromic by the free universal monodromy theorem. Naturally, we want to consider the subgroup of $\Aut(F, E, \varphi)$ corresponding
	to analytically continuing along paths essentially taking  $X$ to $X,$ the correct analogue of the fundamental group.
	We will show that $\Aut(F, E, \varphi)$ is always abelian and in fact isomorphic to the additive group of module homomorphisms 
	from $F/E$ to $\ker \varphi.$ In the analytic continuation set up, $F/E$ corresponds to the first cohomology group and $\ker \varphi$
	will be the space of locally constant functions, with the conclusion being that for various natural classes of functions on tuples of matrices that the corresponding
	fundmental group is abelian.}

	Let $A, B$ be $R$-modules.
	We define $T(A, B)$ to be the $R$-module of module homomorphisms from
	$A$ to $B.$ The following is analogous to the statement that the fundamental group is abelian.
	\begin{theorem}
		Let $F, H$ be $R$-modules. Let $E$ be a submodule of $F.$
		Let $\varphi: F\rightarrow H$ be a module homomorphism. Then,
			$$\Aut(F, E, \varphi)\cong T(F/E, \ker \varphi).$$
	\end{theorem}
	\begin{proof}
		Given $\gamma \in \Aut(F, E, \varphi)$ define the
		map
			$$\phi_\gamma f = f-\gamma f.$$
		Note that
		$$\varphi(f-\gamma f) = \varphi(f)-\varphi(\gamma f)=0,$$
		and therefore
			$\phi_\gamma f \in \ker \varphi.$
			
		Suppose $e \in E.$
		Note that $\phi_\gamma e = e-\gamma e =0.$
		Therefore, we can view
		$\phi_{\gamma}$ as a map from $F/E$ to $\ker \varphi.$
		That is $\phi_{\gamma}\in T(F/E, \ker \varphi).$

		Because $\phi_{\gamma}f \in \ker \varphi,$ we see that
		$\eta \phi_{\gamma}f = \phi_{\gamma}f.$ 
		Now $$\phi_{\eta\gamma}f = f- \eta\gamma f
		= f-\eta f + \eta(f - \gamma f)
		= f-\eta f + f - \gamma f
		= \phi_{\eta}f+ \phi_{\gamma}f.$$
		Therefore, $\phi_{\eta\gamma} = \phi_{\eta} + \phi_{\gamma}.$
		That is, $\phi: \Aut(F, E, \varphi) \rightarrow T(F/E, \ker \varphi)$ is a homomorphism.
		
		To see that $\phi$ is injective, suppose $\phi_{\gamma}f = 0$ for all $f.$
		So $f-\gamma f = 0$ for all $f,$ and therefore $\gamma$ is the identity map.
		
		To see that $\phi$ is surjective, let $T \in T(F/E, \ker \varphi).$
		Define $\gamma_T f = f + Tf.$ Note $\gamma_T \in \Aut(F, E, \varphi).$
	\end{proof}

	The global germs modulo global sections are equivalent to the cohomology group because of the following elementary fact.
	\begin{observation}
		Let $F, H$ be $R$-modules. Let $E$ be a submodule of $F.$
		Let $\varphi: F\rightarrow H$ be a module homomorphism such that $\ker \varphi \subseteq E.$
		Then,
			$$F/E \cong \phi(F)/\phi(E).$$
	\end{observation}

%% file: full.tex
\section{The full fundamental group}

Of course, the results of Section \ref{tracialfund} rely heavily on a differential structure. We now give a construction of a perhaps larger fundamental group which works for not so smooth domains, including commuting
tuples of matrices and other varieties.

Let $D$ be a connected free set.
Consider consider the equivalence relation on paths essentially taking $X$ to $Y$ generated by the relation $\gamma^{\oplus n} = \gamma$ and those induced by fixed endpoint homotopy within $D$.
We define the \dfn{fundamental group} $\pi_1(D)$ to be the set of paths essentially taking $X$ to $X$ up to equivalence. Similarly, we may define the covering space $C(D).$

Note that $\pi^{\tr}_1(D)$ is a quotient of $\pi_1(D).$

\begin{theorem}\label{fullfundamental}
	Let $D$ be a connected free set. Then, $\pi_1(D)$ is abelian and divisible.
\end{theorem}
\begin{proof}
Let $\gamma \in \pi_1(D).$
First note, 
	$$\gamma \oplus \gamma = (\gamma \oplus \gamma_X)(\gamma_X \oplus \gamma).$$
Moreover, 
	$$\gamma \oplus \gamma_X = \gamma_X \oplus \gamma$$
because 
	$$\hat{\gamma}(t,\theta) = \bpm \cos \theta & \sin \theta \\ -\sin \theta & \cos \theta \epm(\gamma \oplus \gamma_X)\bpm \cos \theta & \sin \theta \\ -\sin \theta  & \cos \theta \epm^*$$
gives a fixed endpoint homotopy between the two paths.
So, $$\gamma \oplus \gamma = (\gamma \oplus \gamma_X)^2 = (\gamma_X \oplus \gamma)^2.$$
Therefore, 
	\begin{align*}
		\gamma_1\gamma_2 &= (\gamma_1^2 \oplus \gamma_X)(\gamma_X \oplus \gamma_2^2 ) \\
		&= \gamma_1^2 \oplus \gamma_2^2 \\
		&= (\gamma_2^2 \oplus \gamma_X)(\gamma_X \oplus \gamma_1^2) \\
		&= \gamma_2\gamma_1. 
	\end{align*}
Thus, $\pi(D)$ is abelian.
\end{proof}

We note that free univeral mondromy, as in \cite{pascoeplush} is essentially the statement that $\gamma \oplus \gamma_X = \gamma_X \oplus \gamma$ and therefore holds for thin sets on which there is an appropriate
notion of analytic continuation. (When we analytically continue along the two paths, we much get the same result as they are homotopically equivalent.)
That is, defining the group $\pi_1^{free}(D)$ to be the full fundamental group modulo the global germs of analytic free functions to be the \dfn{free fundamental group}, we see that $\pi_1^{free}(D)$ is trivial.
\begin{observation}
	$\pi_1^{free}(D)$ is the trivial group.
\end{observation}

The more ``purely topological" fundamental group governs analytic continuation in a variety of contexts, such as tracial free functions, determinantal free functions, free functions, possibly permanental analogues,  and the local limits of trace polynomials as in \cite{Jekel}.
Namely, any class of functions on a free set such that $f(X\otimes I) = f(X) \otimes I$ or $f(X\otimes I) = f(X)$ will have analytic continuation governed by the fundamental group.

%% file: questions.tex
\section{Questions}

We close with some questions and conjectures:
\begin{enumerate}
	\item Does $\dim H^1_{\tr}(D) = \rk \pi_1^{\tr}(D)?$
	\item How do we compute the tracial fundamental group and cohomology groups?
	The prototype example would be a set of the form $\{\det p \neq 0\}$ for some free polynomial $p$. We conjecture that $\dim H^1_{\tr}(D),$ $\rk \pi_1^{\tr}(D)$ are both finite in that case.
	\item Develop deformation-type criteria for two domains to be ``homotopic." For example, sufficient conditions for $\pi_1^{\tr}(D_1) = \pi_1^{\tr}(D_2).$
	\item Does nontrivial $H^1_{\tr}(D)$ imply $C^{\tr}(D)$ disconnected?
	\item Can $\pi_1^{\tr}(D)_\mathbb{Z}$ contain a copy of $\mathbb{Q}?$
	\item The proof of free universal monodromy from \cite{pascoeplush} works in somewhat more generality than open free sets.
	The module-based approach says that any family of functions with a derivative-like map into free functions suggests that 
	for thin sets, such as commuting tuples of operators, the tracial fundamental group should have the same structure.
	\item Similar results for functions which are permutation invariant rather than unitarially invariant, such as the permanent of free polynomials should hold with the right notion of analytic continuation.
	Specifically, instead of using open sets along a path, one would need permuation invariant open sets along a path. (One sort of propagates unitary invariance for free, because the unitary group is connected.) Then, one should obtain free universal monodromy and commutativity of the corresponding analog of the tracial fundamental group.
	\item Is $\pi(D)$ a direct sum of copies of $\mathbb{Q}?$
	\item Is $\pi(D)\cong \pi_1^{\tr}(D)?$
\end{enumerate}